\documentclass[12pt]{amsart}
\usepackage{a4wide}

\numberwithin{equation}{section}

\usepackage{verbatim}
\usepackage{mathrsfs}
\usepackage{latexsym}
\usepackage{epsfig}
\usepackage{amsmath}
\usepackage{bbm}
\usepackage{graphics}
\usepackage{amssymb}
\usepackage{amsthm}
\usepackage[alphabetic]{amsrefs}
\usepackage{amsopn}
\usepackage{amscd}
\usepackage[all,knot]{xy}
\xyoption{all}
\SelectTips{cm}{12}
\usepackage{rotating}
\usepackage{hyperref}

\theoremstyle{plain}
\newtheorem{thm}{Theorem}[section]
\newtheorem{lem}[thm]{Lemma}
\newtheorem{prop}[thm]{Proposition}
\newtheorem{cor}[thm]{Corollary}

\newtheorem*{thm*}{Theorem}
\newtheorem*{lem*}{Lemma}
\newtheorem*{prop*}{Proposition}
\newtheorem*{cor*}{Corollary}

\theoremstyle{definition}
\newtheorem{defn}[thm]{Definition}
\newtheorem*{defn*}{Definition}

\newtheorem{ex}[thm]{Example}
{}
\newtheorem{rem}[thm]{Remark}
\newtheorem*{rem*}{Remark}

{}
{}
{}
\newtheorem*{ack}{Acknowledgements}{}

\theoremstyle{remark}
{}
{}
{}

\def\cie{\subseteq}

\def\iso{\cong}

\def\Un{\bigcup}

\def\intersec{\cap}

\def\to{\longrightarrow}

\def\ZZ{\mathbb{Z}}

\def\S{\Sigma}

\def\sfD{\mathsf{D}}

\def\mcK{\mathcal{K}}
\def\mcT{\mathcal{T}}
\def\mcS{\mathcal{S}}
\def\mcI{\mathcal{I}}
\def\mcP{\mathcal{P}}
\def\mcM{\mathcal{M}}

\def\mcL{\mathcal{L}}

\DeclareMathOperator{\Spec}{Spec}

\DeclareMathOperator{\Proj}{Proj}
\DeclareMathOperator{\Spc}{Spc}

\DeclareMathOperator{\supp}{supp}

\DeclareMathOperator{\coker}{coker}

\DeclareMathOperator{\id}{id}
\DeclareMathOperator{\im}{im}
\DeclareMathOperator{\Hom}{Hom}

\DeclareMathOperator{\Ext}{Ext}

\DeclareMathOperator{\cone}{cone}

\DeclareMathOperator{\modu}{\mathsf{mod}}
\DeclareMathOperator{\Modu}{\mathsf{Mod}}
\DeclareMathOperator{\bik}{\mathsf{rel}}
\DeclareMathOperator{\BIK}{\mathsf{Rel}}
\DeclareMathOperator{\stmod}{\mathsf{strel}}
\DeclareMathOperator{\Stmod}{\mathsf{StRel}}
\DeclareMathOperator{\stmodu}{\mathsf{stmod}}

\DeclareMathOperator{\stbik}{\mathsf{strel}}
\DeclareMathOperator{\stBIK}{\mathsf{{StRel}}}

\DeclareMathOperator{\ob}{Ob}

\DeclareMathOperator{\GProj}{GProj}

\DeclareMathOperator{\GInj}{GInj}

\DeclareMathOperator{\proj}{proj}

\DeclareMathOperator{\pdim}{pd}

\DeclareMathOperator{\thick}{thick}

\DeclareMathOperator{\Thick}{\mathrm{Thick}}
\DeclareMathOperator{\Gproj}{\mathrm{Gproj}}
\DeclareMathOperator{\sGproj}{\underline{\mathrm{Gproj}}}

\title[Relative stable categories versus singularity categories]{Comparisons between singularity categories and relative stable categories of finite groups}
\author{Shawn Baland}
\address{Shawn Baland, Department of Mathematics, University of Washington,
Box 354350, Seattle, WA 98195}
\email{sbaland@math.washington.edu}
\author{Greg Stevenson}
\address{Greg Stevenson, Universit\"at Bielefeld, Fakult\"at f\"ur Mathematik, BIREP Gruppe, Postfach 10\,01\,31, 33501 Bielefeld, Germany.}
\email{gstevens@math.uni-bielefeld.de}


\begin{document}

\begin{abstract}
\noindent We consider the relationship between the relative stable category of Benson, Iyengar, and Krause and the usual singularity category for group algebras with coefficients in a commutative noetherian ring. When the coefficient ring is self-injective we show that these categories share a common, relatively large, Verdier quotient. At the other extreme, when the coefficient ring has finite global dimension, there is a semi-orthogonal decomposition, due to Poulton, relating the two categories. We prove that this decomposition is partially compatible with the monoidal structure and study the morphism it induces on spectra.
\end{abstract}

\maketitle

\section{Introduction}
Representations of finite groups are prevalent in mathematics, occurring as important objects of study and providing connections between many areas. Although there are still a great deal of open and challenging problems, the modular representation theory of finite groups over fields is becoming increasingly well understood. In particular, given a finite group $G$ and a field $k$ of characteristic dividing the order of $G$ there has been great progress in understanding the features of $\stmodu kG$, the stable category of finite dimensional representations of $G$. Much of this progress is owed to the existence of a monoidal structure on $\stmodu kG$ which arises from the Hopf algebra structure on $kG$.

On the other hand, the situation is much less satisfactory if one considers representations of $G$ over some general commutative noetherian ring $R$. One immediate obstruction is that there is, in general, no obvious monoidal candidate to replace the stable category. One can still consider the bounded derived category of $\modu RG$, and in some cases this is even equipped with a closed symmetric monoidal structure, but it is convenient to be able to work modulo objects that do not contain interesting modular representation theoretic information. One option is to work with the singularity category of $RG$ but there are reasons to suspect that this might be too brutal a quotient of the derived category.

In \cite{BIK8} Benson, Iyengar, and Krause introduce a version of the stable category of $RG$, denoted here by $\stbik RG$, which retains some information on the modules of finite projective dimension, generalises the construction for field coefficients, and carries a closed symmetric monoidal structure; this last property is important and non-trivial as the tensor product afforded by the Hopf algebra structure on $RG$ is no longer necessarily exact for general coefficient rings. The aim of this paper is to compare, for two classes of coefficient rings, the relative stable category $\stbik RG$ to the singularity category of $RG$. This serves to relate the invariant introduced in \cite{BIK8} to a better understood and more classical one, as well as to shed some light on exactly how $\stbik RG$ extends the singularity category.

The first obvious case to consider is when one takes as coefficient ring a self-injective ring $R$. In this case the group algebra $RG$ is, as in the case for fields, self-injective and so one can consider its usual stable category $\stmodu RG$. In this setting we prove, in Section~\ref{sec_selfinjective}, the following theorem showing that one can compare, and hope to transfer information, between these two categories.

\begin{thm*}[\ref{thm_comparison}]
Let $R$ be a commutative zero dimensional Gorenstein ring, $G$ a finite group, and $\iota\colon R\to RG$ the canonical inclusion. There is an equivalence of triangulated categories
\begin{displaymath}
\phi \colon (\stbik RG)/\mcK' \to (\stmodu RG)/\mcK,
\end{displaymath}
where
\begin{displaymath}
\mcK = \thick_{\stmodu RG}(\iota^*M \; \vert \; M\in \modu R)
\end{displaymath}
and $\mcK'$ is the full subcategory of $\stbik RG$ with the same objects as $\mcK$.
\end{thm*}

As an application we are able to glean some information on the lattice of thick subcategories of $\stbik RG$. This complements the results of \cite{BCS} on the lattice of thick tensor ideals.

In the second part of the paper we consider the case where the coefficient ring $R$ is regular. In this case the singularity category $\sfD_\mathrm{Sg}(RG)$ is a tensor triangulated category and one would morally expect that $\stbik RG$ interpolates between the bounded derived category and the singularity category. In fact this is the case as the following theorem, which is originally due to Poulton \cite{Poulton}, makes precise.

\begin{thm*}[\ref{thm_adjoint}, \ref{cor_adjoint}]
Let $R$ be a commutative noetherian ring of finite global dimension and $G$ a finite group. There is a semi-orthogonal decomposition
\begin{displaymath}
\xymatrix{
\sfD_\mathrm{Sg}(RG) \ar[r]<0.75ex>^-{\underline{\phi}} \ar@{<-}[r]<-0.75ex>_-{\underline{\psi}} & \stmod RG \ar[r]<0.75ex>^-{j^*} \ar@{<-}[r]<-0.75ex>_-{j_*}& \stmod_\mathrm{fpd} RG
}
\end{displaymath}
where $\stmod_\mathrm{fpd}RG$ is the full subcategory consisting of finitely generated modules of finite projective dimension over $RG$.
\end{thm*}

We give a proof of this theorem and explore its compatibility with the monoidal structures on $\sfD_\mathrm{Sg}(RG)$ and $\stmod RG$. We show the adjunction $\underline{\phi} \dashv \underline{\psi}$ satisfies the projection formula and that the morphism of spectra
\begin{displaymath}
\Spc \stmod RG \to \Spc \sfD_\mathrm{Sg}(RG)
\end{displaymath}
induced by $\underline{\phi}$ is dense.

The structure of the paper is straightforward. In the first section we recall from \cite{BIK8} the construction of the relative stable category as well as some facts we will need from Gorenstein homological algebra. In Section~\ref{sec_selfinjective} we prove the first theorem and give some applications to lattices of thick subcategories. Section~\ref{sec_regular} is devoted to the proof of the second theorem and its consequences. Finally, there is an appendix containing a short tensor triangular geometry argument which gives a necessary and sufficient condition for a monoidal functor to induce a dense map on spectra.

\begin{ack}
We are grateful to Andrew Poulton for making us aware of his article \cite{Poulton}, which contains a proof of Theorem~\ref{thm_adjoint}. 
\end{ack}

\section{Preliminaries}\label{sec_prelims}
We provide here some recollections on relative stable categories, as introduced in \cite{BIK8}, and Gorenstein projective modules that will be needed throughout the paper. Before delving into some details on these subjects let us fix some notation. Throughout $R$ will denote a commutative noetherian ring and $G$ will denote a finite group. We will use $\iota$ to denote the canonical inclusion of $R$ into the group algebra $RG$. We recall that $RG$ is a symmetric $R$-algebra. In particular, this means that the extension of scalars functor $\iota^*$ is isomorphic to the coextension of scalars functor $i^!$. The group algebra $RG$ also carries the structure of a cocommutative Hopf algebra endowing $\modu RG$, the category of finitely generated $RG$-modules, with a closed symmetric monoidal structure.

\subsection{Relative stable categories}
In this section we introduce the relative stable category of $RG$ as defined in \cite{BIK8}. We denote, as usual, the abelian category of all $RG$-modules by $\Modu RG$. We say that a short exact sequence
\begin{displaymath}
0\to M' \to M \to M'' \to 0
\end{displaymath}
is $R$\emph{-split} if the short exact sequence
\begin{displaymath}
0 \to \iota_*M' \to \iota_*M \to \iota_*M'' \to 0
\end{displaymath}
is split exact. The category $\Modu RG$ together with the class of $R$-split exact sequences determines an exact category which we denote by $\BIK RG$. In other words, $\BIK RG$ is the category of $RG$-modules but we forget those short exact sequences which are not split as sequences of $R$-modules. Similarly, we denote by $\bik RG$ the category of finitely generated $RG$-modules together with the $R$-split exact structure.

It is proved in \cite{BIK8} that both $\BIK RG$ and $\bik RG$ are Frobenius categories. The projective-injective objects are precisely those in the additive closure of the essential image of $\iota^*$, the extension of scalars functor. Explicitly, an $RG$-module $M$ is projective in $\BIK RG$ if and only if the counit $\iota^*\iota_*M\to M$ is a split epimorphism. We call the projective-injective objects in $\BIK RG$ \emph{weakly projective} modules.

As $\BIK RG$ and $\bik RG$ are Frobenius categories, their stable categories, obtained by factoring out the ideals of maps factoring through projective-injective objects, are triangulated and we denote them by $\stBIK RG$ and $\stbik RG$ respectively. The suspension functor on these categories is given by taking cosyzygies, with respect to the $R$-split exact structure, and the triangles come from $R$-split short exact sequences.

Recall that, using the Hopf algebra structure of $RG$, $\Modu RG$ can be equipped with a closed symmetric monoidal structure. The tensor product is $-\otimes_R -$ with the diagonal action and the internal hom is $\Hom_R(-,-)$ where the action is defined in the usual way using the comultiplication and antipode of $RG$. Although neither $-\otimes_R -$ nor $\Hom_R(-,-)$ is exact in general, both functors are certainly exact on $R$-split exact sequences. Thus both the tensor product and internal hom are exact functors on the exact category $\BIK RG$. Thus they descend to $\stBIK RG$ and $\stbik RG$ endowing these categories with closed symmetric monoidal structures.

\subsection{Gorenstein projectives and the singularity category}
We briefly collect here some definitions and material required in Section~\ref{sec_regular}. A great deal more information on these subjects can be found in \cite{Buchweitzunpub} and \cite{EnochsJenda}. Let us begin with the class of rings which will be of most utility for us.

\begin{defn}
A (not necessarily commutative) noetherian ring $\Lambda$ is \emph{Iwanaga-Gorenstein} (or just Gorenstein for short) if the self-injective dimension of $\Lambda$ over itself is finite on both the left and the right. In this case these injective dimensions have a common value $n$ and we will say $\Lambda$ is $n$-Gorenstein.
\end{defn}

\begin{ex}\label{ex_RG}
If $R$ is an $n$-Gorenstein commutative ring and $G$ is a finite group then the group algebra $RG$ is also $n$-Gorenstein. This result goes back to \cite{EilenbergNakayama}.
\end{ex}

Given a Gorenstein ring $\Lambda$ there are two particularly well behaved classes of modules: the Gorenstein projectives and Gorenstein injectives. We will only speak of the Gorenstein projective modules as we will focus on these, but one can just dualise to obtain the definition of a Gorenstein injective module. Let us remark that both of these notions make sense in greater generality. However, we will make use of the assumption that $\Lambda$ is Gorenstein and not give the most general definition.

\begin{defn}
A right $\Lambda$-module $M$ is \emph{Gorenstein projective} if it occurs as a syzygy in an acyclic complex
\begin{displaymath}
\cdots \to P^i \to P^{i+1} \to \cdots
\end{displaymath}
where each $P^i$ is a projective right $\Lambda$-module.
\end{defn}

There are many equivalent formulations of this condition. The most useful for us are listed in the following theorem, a proof of which can be found, for instance, in \cite{EnochsJenda}.

\begin{thm}
Let $M$ be a right $\Lambda$-module. The following statements are equivalent:
\begin{itemize}
\item[$(1)$] $M$ is Gorenstein projective;
\item[$(2)$] $\Ext^i(M,L)=0$ for all right modules $L$ of finite projective dimension and $i\geq 1$.
\end{itemize}
If $M$ is finitely generated then these conditions are also equivalent to
\begin{itemize}
\item[$(3)$] $\Ext^i(M,P)=0$ for all projective right modules $P$ and $i\geq 1$.
\end{itemize}
\end{thm}

We denote by $\GProj \Lambda$ the full subcategory of Gorenstein projective right $\Lambda$-modules and by $\Gproj \Lambda$ the full subcategory of finitely generated Gorenstein projective right $\Lambda$-modules. Both of these categories inherit exact structures from $\Modu \Lambda$ and are, in fact, Frobenius categories. The projective-injective objects are precisely the projective $\Lambda$-modules.

The stable category $\sGproj RG$ of finitely generated Gorenstein projective modules can be identified with another homological invariant of $\Lambda$, namely the singularity category.

\begin{defn}\label{defn_sg}
The \emph{singularity category} of $\Lambda$ is defined to be the Verdier quotient
\begin{displaymath}
\sfD_{\mathrm{Sg}}(\Lambda) = \sfD^\mathrm{b}(\modu \Lambda)/\sfD^{\mathrm{perf}}(\Lambda),
\end{displaymath}
where $\sfD^\mathrm{b}(\modu \Lambda)$ denotes the bounded derived category of $\modu \Lambda$ and $\sfD^\mathrm{perf}(\Lambda)$ denotes the perfect complexes, i.e.\ the thick subcategory of objects quasi-isomorphic to a bounded complex of finitely generated projectives.
\end{defn}

The connection between these two triangulated categories is given by the following proposition which is originally due to Buchweitz \cite{Buchweitzunpub}.

\begin{prop}\label{prop_sgequiv}
The composite
\begin{displaymath}
\Gproj \Lambda \to \sfD^\mathrm{b}(\modu \Lambda) \to \sfD_\mathrm{Sg}(\Lambda)
\end{displaymath}
induces an equivalence of triangulated categories
\begin{displaymath}
\sfD_\mathrm{Sg}(\Lambda) \iso \sGproj \Lambda
\end{displaymath}
\end{prop}

As noted in Example~\ref{ex_RG} if $R$ is Gorenstein and $G$ is a finite group then the group algebra $RG$ is also Gorenstein. Thus we can apply these constructions to $RG$ and we obtain a triangulated category $\sGproj RG$. The goal of the rest of the paper is to compare the category $\stbik RG$ to $\sGproj RG$. In the classical case, where $R=k$ is a field, every $kG$-module is Gorenstein projective and every sequence of $kG$-modules is $k$-split and these two categories coincide. As we shall see the situation is more complicated in general.

\section{A comparison theorem for self-injective coefficients}\label{sec_selfinjective}

This section is dedicated to comparing the usual stable category and the relative stable category when the coefficients are such that the group algebra is self-injective (so the usual stable category is triangulated). Although the two categories do not necessarily agree we will prove in Theorem~\ref{thm_comparison} that they are ``equivalent up to a thick subcategory''. Stated more precisely, they share a common Verdier quotient.  This allows us to deduce a great deal of information about thick subcategories of $\stbik RG$, especially when $R$ is an artinian complete intersection and $G$ is abelian as explained in Section~\ref{sec_app}.

Throughout this section $R$ will denote a zero dimensional commutative Gorenstein ring, i.e.\ $R$ is artinian and self-injective, and we let $G$ be a finite group. Let $RG$ denote the group algebra and $\iota\colon R\to RG$ the canonical inclusion of $R$. We note that since $R$ is self-injective so is the group algebra $RG$. In fact, it is a Frobenius $R$-algebra.

We have an obvious exact functor $\bik RG \to \modu RG$, namely the identity functor. However, this functor clearly fails to send projectives to projectives. We can attempt to remedy this by defining a thick subcategory of $\stmodu RG$
\begin{displaymath}
\mcK = \thick_{\stmodu RG}(\iota^*M \; \vert \; M\in \modu R)
\end{displaymath}
and considering the composite
\begin{displaymath}
\psi = (\bik RG \to \modu RG \to \stmodu RG \stackrel{\pi}{\to} (\stmodu RG)/\mcK).
\end{displaymath}
The functor $\psi$ kills the projectives in $\bik RG$ by construction and so must factor, as an additive functor, via $\stbik RG$, i.e.\ there is an essentially unique
\begin{displaymath}
\underline{\psi}\colon \stbik RG \to (\stmodu RG)/\mcK
\end{displaymath}
factoring $\psi$. Our theorem will follow from some rumination on the behaviour of $\underline{\psi}$ and $\mcK$.

Before getting under way let us introduce some notation. We will denote (co)syzygies in $\modu RG$ and $\bik RG$ by $\Omega^i$ and $\Omega^i_s$ respectively (the ``s'' here standing for split). For the sake of brevity we write $\mcS = (\stmodu RG)/\mcK$ from this point onward and denote its suspension by $\S_\mcS$. Finally, let us note that we will routinely abuse the isomorphism $i^* \cong i^!$.

\begin{lem}\label{lem_psi_exact}
The functor $\underline{\psi}\colon \stbik RG \to \mcS$ is exact.
\end{lem}
\begin{proof}
For an $RG$-module $M$ we have two short exact sequences in $\modu RG$
\begin{displaymath}
\xymatrix{
0 \ar[r] & M \ar@{=}[d] \ar[r] & \iota^!\iota_*M \ar[r] \ar@{-->}[d]^-{\exists} & \Omega_s^{-1}M \ar[r] \ar@{-->}[d]^-\exists & 0 \\
0 \ar[r] & M \ar[r] & E(M) \ar[r] & \Omega^{-1}M \ar[r] & 0 
}
\end{displaymath}
where $E(M)$ is the injective envelope of $M$, the centre arrow exists by injectivity of $E(M)$, and this extension induces the rightmost vertical arrow. Projecting this diagram to $\mcS$ gives an isomorphism
\begin{displaymath}
\underline{\psi}\Omega_s^{-1}M = \pi\Omega_s^{-1}M \stackrel{\sim}{\to} \pi\Omega^{-1}M = \S_\mcS \underline{\psi}M.
\end{displaymath}
Since the choice of extending morphism we made to construct the original diagram is unique up to $\proj RG$ one can assemble these isomorphisms into a natural isomorphism
\begin{displaymath}
\underline{\psi}\Omega_s^{-1} \stackrel{\sim}{\to} \S_\mcS\underline{\psi}.
\end{displaymath}
It is then essentially immediate that, up to this natural isomorphism, $\underline{\psi}$ sends triangles to triangles as it is the identity on objects and $\bik RG \to \modu RG$ is exact.
\end{proof}

\begin{lem}\label{lem_cone}
If $f\colon M\to N$ is a map of finitely generated $RG$-modules such that
\begin{displaymath}
L = \cone_{\stmodu RG}(f) \in \mcK
\end{displaymath}
then $L' = \cone_{\stbik RG}(f)$, viewed as an object of $\stmodu RG$, lies in $\mcK$.
\end{lem}
\begin{proof}
By definition $L'$ can be constructed (up to projective/injectives) as
\begin{displaymath}
L' = \coker(M \stackrel{(\eta,f)}{\to} \iota^!\iota_*M \oplus N).
\end{displaymath}
This gives rise to a diagram in $\stmodu RG$
\begin{displaymath}
\xymatrix{
M \ar[r]^-\eta \ar[d]_-f & \iota^!\iota_*M \ar[d] \\ 
N \ar[r] \ar[d] & L' \ar[d] \\
L \ar[r]^-{\sim} & L
}
\end{displaymath}
where the top square is homotopy bicartesian and expresses $L'$ as the cokernel described above, and the bottom square arises from completing the vertical maps to triangles (see for instance \cite{NeeCat}*{Lemma~1.4.4}). It follows that $L'\in \mcK$ since $\iota^!\iota_*M$ is in $\mcK$ by definition, $L$ is in $\mcK$ by assumption, and $\mcK$ is thick.
\end{proof}

\begin{lem}\label{lem_K_thick}
The full subcategory
\begin{displaymath}
\mcK' = \{\ob \mcK\}
\end{displaymath}
of $\stbik RG$, i.e.\ the full subcategory on the objects which the weakly projective objects generate in $\stmodu RG$, is thick in $\stbik RG$. In particular, we have an equality
\begin{displaymath}
\mcK' = \underline{\psi}^{-1}(0).
\end{displaymath}
\end{lem}
\begin{proof}
Since $\bik RG \to \modu RG$ is additive and the identity on objects it is immediate that $\mcK'$ is closed under direct sums and summands. If $M\in \mcK'$ then a representative for $\Omega_s^{-1}M$ is
\begin{displaymath}
\coker(M \stackrel{\eta}{\to} \iota^!\iota_*M).
\end{displaymath}
The map $\eta$ is injective so this gives a triangle in $\stmodu RG$ and since $M$ and $\iota^!\iota_*M$ are in $\mcK$ we deduce that the third term of this triangle $\Omega_s^{-1}M$ also lies in $\mcK$. Similar considerations show that $\mcK'$ is also closed under $\Omega_s$.

Finally, suppose $M,N \in \mcK'$ and we are given a map $f\colon M\to N$. By definition $M$ and $N$, viewed in $\stmodu RG$, lie in $\mcK$ and so we know the cone of $f$ taken in $\stmodu RG$ also lies in $\mcK$. Thus we can apply Lemma~\ref{lem_cone} which tells us the cone of $f$ taken in $\stbik RG$ and viewed as an object of $\stmodu RG$ also lies in $\mcK$. Hence it lies in $\mcK'$ and so $\mcK'$ is closed under cones which completes the proof that it is a thick subcategory of $\stbik RG$.

The final statement concerning the kernel then follows: since $\underline{\psi}$ is the identity on objects $\underline{\psi}(M) \cong 0$ in $\mcS$ if and only if $M\in \mcK$ if and only if $\underline{\psi}^{-1}M\in \mcK'$.
\end{proof}

\begin{thm}\label{thm_comparison}
The functor $\underline{\psi}$ induces an equivalence of triangulated categories
\begin{displaymath}
\phi \colon (\stbik RG)/\mcK' \to \mcS = (\stmodu RG)/\mcK,
\end{displaymath}
where
\begin{displaymath}
\mcK = \thick_{\stmodu RG}(\iota^*M \; \vert \; M\in \modu R)
\end{displaymath}
and $\mcK'$ is the full subcategory of $\stbik RG$ with the same objects as $\mcK$.
\end{thm}
\begin{proof}
We have $\underline{\psi}\mcK' = 0$ by construction so $\underline{\psi}$ factors via an exact functor
\begin{displaymath}
\phi \colon (\stbik RG)/\mcK' \to (\stmodu RG)/\mcK.
\end{displaymath}
The functor $\phi$ is still the identity on objects so it is clearly surjective on the nose.

We first show $\phi$ is faithful. Suppose
\begin{displaymath}
\xymatrix{ M \ar@{<-}[r]^-\alpha & L \ar[r]^-\beta & N}
\end{displaymath}
represents a map in $(\stbik RG)/\mcK'$ where $\cone_{\stbik RG}(\alpha)\in \mcK'$. If $\phi$ of this span is zero then there exists a commutative diagram in $\stmodu RG$
\begin{displaymath}
\xymatrix{
& L \ar[dl]_-\alpha \ar[dr]^-\beta & \\
M \ar@{=}[dr] & L' \ar[l]_-{\alpha'} \ar[u] \ar[r] \ar[d] & N \\
& M \ar[ur]_-{0} &
}
\end{displaymath}
where the cone in $\stmodu RG$ of $\alpha'$ lies in $\mcK$. Picking representative for the modules and the morphisms we get a diagram which commutes up to maps factoring through projectives in $\modu RG$. Since every projective $RG$-module is projective in $\bik RG$ this gives a commutative diagram in $\stbik RG$. By Lemma~\ref{lem_cone} the cones on $\alpha$ and $\alpha'$, taken in $\stbik RG$, lie in $\mcK'$ and so this diagram witnesses that our original map was already zero in $(\stbik RG)/\mcK'$.

We now show $\phi$ is full. Let
\begin{displaymath}
\xymatrix{ M \ar@{<-}[r]^-\alpha & L \ar[r]^-\beta & N}
\end{displaymath}
represent a map in $(\stmodu RG)/\mcK$, i.e.\ the object $\cone_{\stmodu RG}(\alpha)$ lies in $\mcK$. Again using Lemma~\ref{lem_cone} we have that $\cone_{\stbik RG}(\alpha)$ lies in $\mcK'$ so the same span represents a morphism in $(\stbik RG)/\mcK'$. As $\phi$ is just the identity on objects and honest maps of modules this shows $\phi$ is full.
\end{proof}

As an immediate consequence of the theorem we can identify certain sublattices of thick subcategories in the relative and usual stable categories. For a triangulated category $\mcT$ we denote by $\Thick(\mcT)$ the lattice, ordered by inclusion, of thick subcategories of $\mcT$.

\begin{cor}\label{cor_comparison}
The equivalence $\phi$ induces a lattice isomorphism
\begin{displaymath}
\{\mcL \in \Thick(\stbik RG) \; \vert \; \mcK' \subseteq \mcL\} \cong \{\mcM \in \Thick(\stmodu RG) \; \vert \; \mcK \subseteq \mcM\}
\end{displaymath}
\end{cor}

It is worth pointing out that we can actually express $\mcK$ in a slightly more compact form. The ring $RG$ is finite free over $R$ and so $\iota^*$ is an exact functor. It thus induces an exact functor
\begin{displaymath}
\underline{\iota}^* \colon \stmodu R \to \stmodu RG.
\end{displaymath}
Using this observation we come to the following conclusion.

\begin{lem}\label{lem_whatisK?}
Suppose that $R$ is local with residue field $k$. Then there is an equality of thick subcategories
\begin{displaymath}
\mcK = \thick_{\stmodu RG}(kG).
\end{displaymath}
\end{lem}
\begin{proof}
Since $R$ is local it is clear that
\begin{displaymath}
\stmodu R = \thick_{\stmodu R}(k).
\end{displaymath}
Thus $\mcK$, the thick subcategory generated by the essential image of $\underline{\iota}^*$, is just the thick subcategory generated by $\iota^*k = kG$.
\end{proof}

\subsection{An application: complete intersection coefficients}\label{sec_app}

We now suppose that we have a local artinian complete intersection $(R,\mathfrak{m},k)$ and outline what Corollary~\ref{cor_comparison} tells us about thick subcategories of $\stbik RG$ when $G$ is abelian. We will use terms, such as support, without defining them as this would take us somewhat far afield. Precise definitions can be found in \cite{Stevensonclass}. Our approach relies on the following fact, whose proof we omit since it is rather straightforward.

\begin{lem}
If $R$ is a complete intersection and $G$ is abelian then the group algebra $RG$ is a complete intersection.
\end{lem}

In this case, by the following theorem, we understand $\stmodu RG$ fairly well.

\begin{thm}[\cite{Stevensonclass}]\label{thm_me}
There are order preserving bijections
\begin{displaymath}
\Thick(\stmodu RG)
\xymatrix{ \ar[r]<1ex>^{\sigma} \ar@{<-}[r]<-1ex>_{\tau} &} \left\{
\begin{array}{c}
\text{specialisation closed} \\ \text{subsets of $\mathbb{P}^{c-1}_k$}
\end{array} \right\}
\end{displaymath}
where $c$ is the codimension of $RG$.
\end{thm}

So we can apply Corollary~\ref{cor_comparison} to deduce the following theorem concerning the relative stable category.

\begin{thm}\label{thm_app}
Suppose $R$ is a local artinian complete intersection with residue field $k$ and $G$ is abelian. There are order preserving bijections
\begin{displaymath}
\Thick(\stbik RG / \mcK')
\xymatrix{ \ar[r]<1ex>^{\sigma} \ar@{<-}[r]<-1ex>_{\tau} &} \left\{
\begin{array}{c}
\text{specialisation closed} \\ \text{subsets of $\mathbb{P}^{c-1}_k \setminus \mathcal{V}$}
\end{array} \right\}
\end{displaymath}
where $c$ is the codimension of $RG$ and $\mathcal{V}$ is the closed subset of $\mathbb{P}^{c-1}_k$ corresponding to the module $kG$, i.e.\ its support.
\end{thm}
\begin{proof}
Together Theorem~\ref{thm_me} and Corollary~\ref{cor_comparison} buy us that there is a bijection between the left hand side and specialisation closed subsets of $\mathbb{P}^{c-1}_k$ containing $\sigma\mcK$. Such subsets are the same as specialisation closed subsets of the complement of $\sigma\mcK$. It just remains to identify $\sigma\mcK$ with the support of $kG$. This is an immediate consequence of Lemma~\ref{lem_whatisK?}.
\end{proof}

Let us now consider a more concrete example to show that the lattice of thick subcategories of $\stbik RG$ can be quite large, especially in comparison to the spectrum which was computed in \cite{BCS}. Let us take $R=\ZZ/p^n\ZZ$ for our coefficient ring and let $G$ be an elementary abelian $p$-group of rank $r$. In this case $RG$ has codimension $r+1$. The $RG$-module $kG$ is clearly indecomposable and has complexity $1$ and so its support, the subset $\mathcal{V}$ of the theorem, is a single point $\lambda$. Thus one can embed the specialisation closed subsets of $\mathbb{P}^r_k \setminus \{\lambda\}$ into the lattice of thick subcategories of $\stbik RG$.

In particular, if we take $r=1$ and $n=2$, we have a copy of the specialisation closed subsets of $\mathbb{A}^1_k$ in $\Thick(\stbik RG)$. This is in rather stark contrast to the results of \cite{BCS} which show that the lattice of radical thick tensor ideals in $\stbik RG$ has $4$ elements.

\section{A semi-orthogonal decomposition for regular coefficients}\label{sec_regular}
We now consider the opposite extreme of coefficients, namely commutative noetherian rings of finite global dimension. In this case $\stbik RG$ admits a semi-orthogonal decomposition involving the stable category of Gorenstein projective (aka maximal Cohen-Macaulay) $RG$-modules, as was proved by Poulton \cite{Poulton}. This gives a precise sense in which $\stbik RG$ is an extension of the usual singularity category by certain modules of finite projective dimension. We give a complete proof of this theorem and then present some consequences which appear to be new.

Throughout this section $R$ denotes a regular ring of finite global dimension $d$, $G$ denotes a finite group, and $\iota$ denotes the inclusion $R\to RG$. Given these hypotheses $RG$ is $d$-Gorenstein. As in the previous section we will continue to abuse, without explicit mention, the isomorphism of functors $\iota^*\cong \iota^!$. 

\begin{lem}\label{lem_gproj_pres}
The functor $\iota^*$ sends modules of finite projective (injective) dimension to modules of finite projective (injective) dimension. Thus for all $M\in \Modu R$ we have $\pdim_{RG}\iota^*M \leq d$. It follows that $\iota_*$ sends Gorenstein projective (injective) $RG$-modules to Gorenstein projective (injective) $R$-modules, which are precisely the projective (injective) $R$-modules.
\end{lem}
\begin{proof}
The functor $\iota^*$ is exact and preserves projectives and injectives since it is both the left and right adjoint of an exact functor. Thus it sends projective (injective) resolutions to projective (injective) resolutions. Since $R$ is regular every module in the image of $\iota^*$ thus has finite projective (injective) dimension over $RG$.

For the last statement suppose $A \in \GProj RG$ and $M\in \Modu R$. We have seen above that $\iota^*M$ has finite projective dimension over $RG$. So for $i>0$ the groups $\Ext_{RG}^i(A, \iota^* M)$ vanish. Thus for $i>0$
\begin{displaymath}
0 = \Ext_{RG}^i(A, \iota^* M) \iso \Ext_R^i(\iota_* A, M)
\end{displaymath}
showing that $\iota_* A$ is in $\GProj R = \Proj R$. The statement for Gorenstein injectives is proved similarly.
\end{proof}

As usual we can consider $\GProj RG$ and $\GInj RG$ as full exact subcategories of $\Modu RG$. They are both Frobenius categories. We choose to focus on the case of Gorenstein projectives, so we can work with finitely generated modules, and we will state and prove our results for $\GProj RG$. However, it is worth noting that most of what we do is also, once suitably modified by taking duals, valid for Gorenstein injectives.

\begin{lem}\label{lem_gpi_split}
Every short exact sequence in $\GProj RG$ is $R$-split.
\end{lem}
\begin{proof}
Suppose $0 \to A' \to A \to A'' \to 0$ is a short exact sequence of Gorenstein projective $RG$-modules. By the last lemma $0 \to \iota_* A' \to \iota_* A \to \iota_* A'' \to 0$ is an exact sequence of projective $R$-modules and so certainly splits.
\end{proof}

We will also need the following easy observation.

\begin{lem}\label{lem_cw_fpd}
If $X\in \modu RG$ is weakly projective then $\pdim_{RG} X < \infty$.
\end{lem}
\begin{proof}
If $X$ is weakly projective then the counit $\iota^*\iota_*X \to X$ is a split epimorphism by \cite{BIK8}*{Theorem~2.6}. Thus
\begin{displaymath}
\pdim_{RG} X \leq \pdim_{RG} \iota^*\iota_* X \leq d,
\end{displaymath}
where $\iota^*\iota_*X$ has finite projective dimension by Lemma~\ref{lem_gproj_pres}.
\end{proof}

We now begin to compare $\bik RG$ to the Frobenius category $\Gproj RG$ of finitely generated Gorenstein projective $RG$-modules.

\begin{lem}\label{lem_gproj_inclusion}
The inclusion $\phi\colon \Gproj RG \to \bik RG$ is exact and sends projective-injective objects to projective-injective objects.
\end{lem}
\begin{proof}
By Lemma~\ref{lem_gpi_split} every exact sequence in $\Gproj RG$ is $R$-split so $\phi$ is exact. The projectives in $\Gproj RG$ are the $RG$-projectives and these are certainly weakly projective.
\end{proof}

\begin{rem}
The collections $\GProj RG$ and $\GInj RG$ form exact Frobenius subcategories of $\BIK RG$, the category of all $RG$-modules with the $R$-split exact structure, by essentially the same argument.
\end{rem}

Thus $\Gproj RG$ is an exact Frobenius subcategory of $\bik RG$ and so $\phi$ induces an exact functor
\begin{displaymath}
\underline{\phi} \colon \sGproj RG \to \stmod RG.
\end{displaymath}

It is clear that $\Gproj RG$ is closed under admissible extensions and kernels of admissible epimorphisms in $\modu RG$. However, it is not immediate that it is closed under cokernels of admissible monomorphisms. We show that this is the case\textemdash{}it follows from a standard result concerning Gorenstein projectives over $RG$.

\begin{prop}\label{prop_gproj_criterion}
Let $M$ be a finitely generated $RG$-module. Then $M$ is Gorenstein projective if and only if $\iota_*M$ is projective over $R$.
\end{prop}
\begin{proof}
We have already seen in Lemma~\ref{lem_gproj_pres} that $M$ is Gorenstein projective only if $\iota_*M$ is projective.

So let us assume $\iota_*M$ is projective and prove $M$ is Gorenstein projective. It is enough, by \cite{EnochsJenda}*{Corollary~11.5.3}, to check that $\Ext^i_{RG}(M,RG) = 0$ for $i\geq 1$. For $\mathfrak{p}\in \Spec R$ we have isomorphisms
\begin{displaymath}
\Ext^i_{RG}(M,RG)_\mathfrak{p} \iso \Ext^i_{R_\mathfrak{p}G}(M_\mathfrak{p}, R_\mathfrak{p}G)
\end{displaymath}
so we may reduce to the case $R$ is local. We will prove the statement by induction on the dimension of the regular local ring $R$.

If $\dim R = 0$ then $R$ is a field and the statement is immediate. Suppose then the result holds for regular local rings of dimension strictly less than $d$ and let $\dim R = d>0$. As $R$ is regular we may choose a regular element $r\in \mathfrak{m}\setminus \mathfrak{m}^2$. As both $RG$ and $M$ are free over $R$ the element $r$ is also a non-zerodivisor on these modules. Set $R' = R/rR$ and $M' = M/rM$. By the inductive hypothesis $M'$ is a Gorenstein projective $R'G$-module so
\begin{displaymath}
\Ext^i_{RG}(M, R'G) \iso \Ext^i_{R'G}(M',R'G) = 0 \quad \text{for} \; i>0,
\end{displaymath}
where the isomorphism of $\Ext$ groups is a result of $M$ being acyclic for the base change functor $R'\otimes_R(-)$ as $M$ is $R$-free. Considering the long exact sequence for $\Ext$ coming from
\begin{displaymath}
0 \to RG \stackrel{r}{\to} RG \to R'G \to 0
\end{displaymath}
we see multiplication map $\Ext^i_{RG}(M,RG) \stackrel{r}{\to} \Ext^i_{RG}(M,RG)$ is surjective for $i>0$. Since these $\Ext$s are finitely generated $R$-modules they must vanish by Nakayama's lemma.
\end{proof}


\begin{cor}
The full subcategory $\Gproj RG$ is closed under cokernels of admissible monomorphisms in $\bik RG$.
\end{cor}
\begin{proof}
This is immediate from the proposition: if $0\to A'\to A$ is $R$-split and $A',A$ are Gorenstein projective, hence $R$-projective, then the cokernel is also $R$-projective and hence Gorenstein projective.
\end{proof}

We next show $\underline{\phi}$ is an embedding.

\begin{prop}
The functor $\underline{\phi}$ is fully faithful.
\end{prop}
\begin{proof}
Let $A$ and $B$ be finitely generated Gorenstein projective $RG$-modules. We will show that $f\colon A\to B$ factors through a weakly projective module if and only if it factors through a projective module. Since projectives are weakly projective the `if' direction is clear.

Suppose $f$ factors via a weakly projective module. Then by \cite{BIK8}*{Lemma~5.4} it factors via the counit $\iota^*\iota_* B \to B$. By Lemma~\ref{lem_gproj_pres} the module $\iota_*B$ is $R$-projective, so $\iota^*\iota_*B$ is $RG$-projective and thus $f$ factors via a projective.
\end{proof}

Thus we can view $\sGproj RG$ as a thick subcategory of $\stmod RG$ via $\underline{\phi}$.  We next produce an adjoint for $\underline{\phi}$. In order to do so we will make use of the singularity category $\sfD_{\mathrm{Sg}}(RG)$ of $RG$ (see Definition~\ref{defn_sg}). 

We consider the composite
\begin{displaymath}
\psi = (\bik RG \to \sfD^b(\modu RG) \to \sfD_\mathrm{Sg}(RG) \stackrel{\sim}{\to} \sGproj RG ),
\end{displaymath}
where the final equivalence is given by Proposition~\ref{prop_sgequiv}. Of course every $R$-split exact sequence, i.e.\ every exact sequence in $\bik RG$, is exact in the abelian category $\modu RG$ and so gives a triangle in $\sfD^b(\modu RG)$. Thus $\psi$ sends $R$-split exact sequences to triangles in $\sGproj RG$. By Lemma~\ref{lem_cw_fpd} if $X\in \modu RG$ is weakly injective then $\pdim_{RG}X < \infty$ so $\psi(X) = 0$. Thus $\psi$ induces an exact functor
\begin{displaymath}
\underline{\psi} \colon \stmod RG \to \sGproj RG.
\end{displaymath}

\begin{rem}
It is essentially immediate from the construction that $\underline{\psi}\,\underline{\phi} \iso \id_{\sGproj R}$. Indeed, this composite is induced by
\begin{displaymath}
\Gproj RG \to \bik RG \to \sfD^b(\modu RG) \to \sfD_\mathrm{Sg}(RG) \stackrel{\sim}{\to} \sGproj RG
\end{displaymath}
which is just the usual projection from $\Gproj RG$ onto its stable category.
\end{rem}

Before proving that $\underline{\psi}$ is right adjoint to $\underline{\phi}$ we remind the reader of another fact from Gorenstein homological algebra.

\begin{thm}[\cite{EnochsJenda}*{Theorem~11.5.1}]
Every $RG$-module $M$ has a Gorenstein projective precover $A\to M$ fitting into a short exact sequence
\begin{displaymath}
0 \to L \to A \to M \to 0
\end{displaymath}
such that $L$ has finite projective dimension.
\end{thm}

We would like to use this result to exhibit the triangles corresponding to the claimed adjunction between $\underline{\phi}$ and $\underline{\psi}$. However, the short exact sequence in the above theorem is not necessarily $R$-split, i.e.\ not necessarily exact in $\BIK RG$. This is relatively easily fixed as we show in the next lemma.

\begin{lem}\label{lem_gproj_splitcover}
Let $M$ be an $RG$-module and $A\to M$ a Gorenstein projective precover as in the theorem. Then the epimorphism $A\oplus \iota^*\iota_*M \to M$, induced by the precover and the counit, fits into an $R$-split exact sequence
\begin{displaymath}
0 \to L' \to A\oplus \iota^*\iota_* M \to M \to 0
\end{displaymath}
where $\pdim_{RG}L' < \infty$.
\end{lem}
\begin{proof}
The map $A\oplus \iota^*\iota_*M \to M$ is $R$-split since $\iota^*\iota_*M \to M$ is $R$-split. Thus the only real content to the statement is that the kernel has finite projective dimension. Since $\iota^*\iota_*M$ has finite projective dimension (by Lemma~\ref{lem_cw_fpd}) the morphisms $A\oplus \iota^*\iota_*M \to M$ and $A\to M$ agree in the singularity category. Since $A\to M$ is an isomorphism in the singularity category we see that $A\oplus \iota^*\iota_*M \to M$ must also be an isomorphism in the singularity category. Hence the mapping cone in $\sfD^b(\modu RG)$, namely $\S L'$, is perfect so $L'$ has finite projective dimension.
\end{proof}

\begin{lem}\label{lem_gproj_orthog}
Let $M\in \Modu RG$ have finite projective dimension. Then for every $A\in \GProj RG$ we have $\Stmod(A,M) = 0$.
\end{lem}
\begin{proof}
Since $M$ has finite projective dimension we can choose an exact sequence
\begin{displaymath}
0 \to M' \to P \to M \to 0
\end{displaymath}
where $P$ is projective and $M'$ has finite projective dimension. Applying $\Hom_{RG}(A,-)$ yields an exact sequence
\begin{displaymath}
0 \to \Hom_{RG}(A,M') \to \Hom_{RG}(A,P) \to \Hom_{RG}(A,M) \to 0
\end{displaymath}
since $\Ext^1_{RG}(A,M')=0$ by virtue of $M'$ having finite projective dimension. Thus every map from $A$ to $M$ factors via the weakly projective module $P$ and so $\Stmod(A,M) = 0$ as claimed.
\end{proof}

\begin{thm}[\cite{Poulton}*{Theorem~3.5}]\label{thm_adjoint}
The functor $\underline{\psi}\colon \stmod RG \to \sGproj RG$ is right adjoint to the fully faithful inclusion $\underline{\phi}$.
\end{thm}
\begin{proof}
Let $M$ be an object of $\stmod RG$. By Lemma~\ref{lem_gproj_splitcover} we can find a triangle
\begin{displaymath}
\underline{\phi}A \to M \to L \to \S A
\end{displaymath}
in $\stmod RG$ where $L$ is isomorphic to a module of finite projective dimension. By the last lemma $L$ is thus in $\sGproj RG^\perp$ and by abstract nonsense it follows that $\underline{\phi}$ has a right adjoint given by sending $M$ to $A$ (the abstract statement can be found, for instance, in \cite{BondalReps}*{Lemma~3.1}). Since $L$ has finite projective dimension applying $\underline{\psi}$ to this triangle gives $A\iso \underline{\psi}\, \underline{\phi}A \iso \underline{\psi}M$. A standard argument then shows that the right adjoint we have produced by nonsense is isomorphic to $\underline{\psi}$. 
\end{proof}

We can easily determine the quotient of $\stmod RG$ by $\sGproj RG$. 

\begin{lem}\label{lem_kerneldescription}
There are equivalences of triangulated categories
\begin{displaymath}
\stmod RG / \sGproj RG \iso \sGproj RG^\perp = \ker\underline{\psi} = \stmod_\mathrm{fpd} RG,
\end{displaymath}
where
\begin{displaymath}
\stmod_\mathrm{fpd} RG = \{ M\in \stmod RG \; \vert \; \pdim_{RG} M < \infty\}.
\end{displaymath}
\end{lem}
\begin{proof}
The first two equivalences are standard and follow from the existence of the right adjoint $\underline{\psi}$ of $\underline{\phi}$. 
To see the third we just note that, by definition, the kernel of $\underline{\psi}$ is precisely the class of modules of finite projective dimension; there is no need to close under isomorphism as one easily checks that a module of infinite projective dimension cannot become isomorphic to a module of finite projective dimension in $\stmod RG$. 
\end{proof}

The following corollary to Theorem~\ref{thm_adjoint} summarises what we have proved about the semi-orthogonal decomposition so far.

\begin{cor}\label{cor_adjoint}
Let $R$ be a commutative noetherian ring of finite global dimension and $G$ a finite group. Then there is a semi-orthogonal decomposition
\begin{displaymath}
\xymatrix{
\sGproj RG \ar[r]<0.75ex>^-{\underline{\phi}} \ar@{<-}[r]<-0.75ex>_-{\underline{\psi}} & \stmod RG \ar[r]<0.75ex>^-{j^*} \ar@{<-}[r]<-0.75ex>_-{j_*}& \stmod_\mathrm{fpd} RG
}
\end{displaymath}
where $\stmod_\mathrm{fpd}RG$ is the full subcategory consisting of finitely generated modules of finite projective dimension over $RG$.
\end{cor}

\subsection{Compatibility with the tensor structure}

The category $\stmod RG$ is tensor triangulated. It is thus natural to ask if the localization sequence we have produced is compatible with this tensor structure. Compatibility with the inclusion of $\sGproj RG$ is easy based upon what we have shown so far.


\begin{cor}
The category $\sGproj RG$ is closed symmetric monoidal and the inclusion $\underline{\phi}$ is closed monoidal, i.e.\ $\sGproj RG$ is a tensor subcategory of $\stmod RG$.
\end{cor}
\begin{proof}
The monoidal structure on $\sGproj RG$ is given by restricting $\otimes_R$ and $\Hom_R$, the closed monoidal structure given by the Hopf algebra structure on $RG$. This makes sense as by Proposition~\ref{prop_gproj_criterion} the trivial module $R$ is Gorenstein projective and if $M$ and $M'$ are Gorenstein projective then so are $M\otimes_R M'$ and $\Hom_R(M,M')$. Since $\underline{\phi}$ is induced by an honest inclusion it is clear it is closed monoidal, in fact strictly so. 
\end{proof}

Thus $\underline{\phi}$ induces a map of spectra in the sense of \cite{BaSpec}
\begin{displaymath}
\Spc(\underline{\phi})\colon \Spc(\stmod RG) \to \Spc(\sGproj RG).
\end{displaymath}

We now study this situation a little more closely. This is warranted as Balmer's theory provides powerful tools for studying tensor triangulated categories. Moreover, one expects $\sGproj RG$ to be easier to understand than $\stmod RG$ and so $\Spc\underline{\phi}$ should be a useful tool in trying to understand the spectrum of $\stmod RG$.

We refer to \cite{BaSpec} for the relevant definitions. Let us just recall that a thick subcategory $\mcI$ is a \emph{tensor ideal} in $\sGproj RG$ (resp.\ $\stmod RG$) if it is closed under tensoring with arbitrary objects of $\sGproj RG$ (resp.\ $\stmod RG)$ and a proper thick tensor ideal $\mcP$ is \emph{prime} if $M\otimes_R N \in \mcP$ implies that at least one of $M$ or $N$ lies in $\mcP$. 

\begin{lem}
The tensor triangulated category $\sGproj RG$ is rigid, i.e.\ setting $M^\vee = \Hom_R(M,R)$ the natural map
\begin{displaymath}
\gamma_{M,N}\colon M^\vee \otimes_R N \to \Hom_R(M, N)
\end{displaymath}
is an isomorphism for all $M$ and $N$ in $\sGproj RG$.
\end{lem}
\begin{proof}
Applying the restriction functor $\iota_*$ to $\gamma_{M,N}$ yields an isomorphism of $R$-modules, in fact it is just the usual duality morphism for finitely generated projectives. Since $\iota_*$ is conservative we see $\gamma_{M,N}$ itself is an isomorphism.
\end{proof}

\begin{prop}\label{prop_projformula}
The adjunction $\underline{\phi} \dashv \underline{\psi}$ satisfies the projection formula: for all $M \in \sGproj RG$ and $X\in \stmod RG$ there is a natural isomorphism
\begin{displaymath}
M\otimes_R \underline{\psi}X \stackrel\sim\to \underline{\psi}(\underline{\phi}M\otimes_R X).
\end{displaymath}
\end{prop}
\begin{proof}
Such a natural morphism always exists by formal nonsense and is an isomorphism by rigidity of $\sGproj RG$, see for instance \cite{FauskHuMay}*{Proposition~3.12}.
\end{proof}

As an immediate and perhaps somewhat surprising corollary we observe the following.

\begin{cor}\label{cor_action}
Let $L$ be an $RG$-module of finite projective dimension and $M\in \GProj RG$. Then $M\otimes_R L$ has finite projective dimension as an $RG$-module. In particular $\sGproj RG$ acts on $\stmod_\mathrm{fpd} RG$.
\end{cor}
\begin{proof}
We just note that
\begin{displaymath}
\underline{\psi}(M\otimes_R L) = \underline{\psi}(\underline{\phi}M\otimes_R L)\cong M\otimes_R \underline{\psi}L \cong M\otimes_R 0 \cong 0,
\end{displaymath}
where we have used the projection formula and that $L$ is in the kernel of $\underline{\psi}$ by Lemma~\ref{lem_kerneldescription}. Again using Lemma~\ref{lem_kerneldescription} we deduce that $M\otimes_R L$ has finite projective dimension.
\end{proof}

The projection formula also allows us some level of control over the process of completing a thick tensor ideal of $\sGproj RG$ to a thick tensor ideal of $\stmod RG$. We now make this precise and as a consequence show $\Spc(\underline{\phi})$ is dense.

\begin{lem}\label{lem_control}
Let $\mcI$ be a thick tensor ideal of $\sGproj RG$ and let $\widetilde{\mcI}$ denote the thick tensor ideal of $\stmod RG$ generated by $\underline{\phi}\mcI$. Then we have an equality (up to closing under isomorphisms)
\begin{displaymath}
\underline{\psi}\widetilde{\mcI} = \mcI.
\end{displaymath}
\end{lem}
\begin{proof}
We begin by observing that 
\begin{align*}
\widetilde{\mcI} &= \thick^\otimes_{\stmod RG}(\underline{\phi}\mcI) \\
&= \thick_{\stmod RG}(N\otimes \underline{\phi}M \; \vert \; N\in \stmod RG \; \text{and} \; M\in \mcI),
\end{align*}
where the second equality follows from (the thick subcategory version of) \cite{StevensonActions}*{Lemma~3.10}. Now we can use the projection formula to see that
\begin{align*}
\underline{\psi}\widetilde{\mcI} &= \underline{\psi}\thick_{\stmod RG}(N\otimes \underline{\phi}M \; \vert \; N\in \stmod RG \; \text{and} \; M\in \mcI) \\
&\subseteq \thick_{\stmod RG}(\underline{\psi}(N\otimes \underline{\phi}M) \; \vert \; N\in \stmod RG \; \text{and} \; M\in \mcI) \\
&= \thick_{\stmod RG}(\underline{\psi}N\otimes M \; \vert \; N\in \stmod RG \; \text{and} \; M\in \mcI) \\
&\subseteq \mcI,
\end{align*}
where the last containment uses that $\mcI$ is an ideal. This is in fact an equality because fully faithfulness of $\underline{\phi}$ gives
\begin{displaymath}
\mcI = \underline{\psi}\,\underline{\phi}\mcI \subseteq \underline{\psi}\widetilde{\mcI} \subseteq \mcI
\end{displaymath}
up to closing under isomorphisms.
\end{proof}

\begin{prop}\label{prop_denseapp}
The morphism $\Spc(\underline{\phi})\colon \Spc(\stmod RG) \to \Spc(\sGproj RG)$ induced by $\underline{\phi}$ is dense.
\end{prop}
\begin{proof}
By Proposition~\ref{prop_dense1} it is enough to check that if $\underline{\phi}M$ generates $\stmod RG$ as a tensor ideal then $M$ generates $\sGproj RG$ as a tensor ideal. This follows immediately from Lemma~\ref{lem_control}: if the thick tensor ideal generated by $\underline{\phi}M$ is $\stmod RG$ then $\underline{\psi}\stmod RG = \sGproj RG$ must be contained in the thick tensor ideal $M$ generates in $\sGproj RG$.
\end{proof}

The right orthogonal $\stmod_\mathrm{fpd} RG$ does not seem to be compatible with the monoidal structure beyond the action of $\sGproj RG$ produced in Corollary~\ref{cor_action}. As the following example shows it is not even necessarily closed under taking tensor powers.

\begin{ex}
We let $R=\ZZ$ and $G=C_2$, the cyclic group of order $2$. Thus we are dealing with the hypersurface $\ZZ[x]/(x^2-1)$. Consider the module $M$ of finite projective dimension defined by the presentation
\begin{displaymath}
\xymatrix{
0 \ar[r] & \ZZ C_2 \ar[r]^-{x-3} & \ZZ C_2 \ar[r] & M \ar[r] & 0.
}
\end{displaymath}
More explicitly, $M$ is $\ZZ/8\ZZ$ as an abelian group with $x$ acting as multiplication by $3$. A straightforward computation using \cite{BIK8}*{Theorem~2.6(5)} reveals that $M$ is not weakly projective and so defines a non-zero object of $\stmod_\mathrm{fpd} \ZZ C_2$.

We now consider the module $N = M\otimes_\ZZ M$. As an abelian group this is again $\ZZ/8\ZZ$ but now with the trivial action of $x$. The non-split extension
\begin{displaymath}
\xymatrix{
0 \ar[r] & \ZZ \ar[r]^-8 & \ZZ \ar[r] & N \ar[r] & 0,
}
\end{displaymath}
where $\ZZ$ is the trivial module, implies $N$ has infinite projective dimension (and so in particular is not weakly projective) via a straightforward cohomology computation. 
\end{ex}

\appendix
\section{A brief aside on denseness in tensor triangular geometry}

This short appendix exists to serve Proposition~\ref{prop_denseapp}. Let $\mcK$ and $\mcL$ be essentially small tensor triangulated categories and $F\colon \mcK \to \mcL$ an exact monoidal functor. We will denote the tensor product in both $\mcK$ and $\mcL$ by $\otimes$, which we hope will not cause confusion as it should always be abundantly clear from the context in which category tensor products are being taken.

In \cite{BaSpec}*{Proposition~3.9} Balmer characterises the closure of the image of $\widetilde{F} = \Spc(F)$ in terms of objects of $\mcK$ whose images under $F$ generate $\mcL$ as a tensor ideal. We provide here some trivial arguments to extract, from Balmer's general result, a necessary and sufficient condition for $\widetilde{F}$ to be dense.

We will consider the following two sets (up to taking isomorphism classes) of objects of $\mcK$
\begin{displaymath}
\mathcal{S}_{\mcK} = \{a\in \mcK \; \vert \; \thick^\otimes(a) = \mcK\}
\end{displaymath}
and
\begin{displaymath}
\mathcal{S}_{\mcL} = \{a\in \mcK \; \vert \; \thick^\otimes(Fa) = \mcL\}
\end{displaymath}
i.e.\ the sets of objects in $\mcK$ which are supported everywhere and whose images are supported everywhere respectively. 

\begin{lem}\label{lem_dense_containments}
The sets defined above satisfy $\mathcal{S}_\mcK \subseteq \mathcal{S}_\mcL$.
\end{lem}
\begin{proof}
This is an immediate consequence of \cite{BaSpec} Corollary 2.5 and Proposition 3.6. Indeed, for any $a\in \mathcal{S}_\mcK$
\begin{displaymath}
\supp_{\mcL}(Fa) = \widetilde{F}^{-1}(\supp_{\mcK}(a)) = \widetilde{F}^{-1}(\Spc \mcK) = \Spc\mcL
\end{displaymath}
so that we must have $\thick^\otimes(Fa) = \mcL$ showing $a\in \mathcal{S}_\mcL$.
\end{proof}

Recall from \cite{BaSpec}*{Proposition~3.9} that 
\begin{displaymath}
\overline{\im\widetilde{F}} = Z(\mathcal{S}_{\mcL}) = \{\mathcal{P}\in \Spc\mcK \; \vert \; \mathcal{P}\intersec\mathcal{S}_{\mcL} = \varnothing \}
\end{displaymath}
so that $\widetilde{F}$ is dense if and only if $\overline{\im\widetilde{F}} = \Spc\mcK$, if and only if 
\begin{equation}\label{eq_dense}
Z(\mathcal{S}_\mcL) = \Spc\mcK.
\end{equation}
Our condition boils down to the following trivial generalization of \cite{BaSpec}*{Corollary~2.5}.

\begin{lem}\label{lem_cor2.5_bis}
A set of objects $\mathcal{S} \cie \mcK$ satisfies $Z(\mathcal{S}) = \Spc\mcK$ if and only if every $a\in \mathcal{S}$ satisfies $\thick^\otimes(a) = \mcK$.
\end{lem}
\begin{proof}
Observe that $Z(\mathcal{S}) = \Spc\mcK$ if and only if $U(\mathcal{S}) = \Spc\mcK\setminus Z(\mathcal{S}) = \varnothing$. Now 
\begin{displaymath}
\varnothing = U(\mathcal{S}) = \Un_{a\in \mathcal{S}} U(a)
\end{displaymath}
if and only if for each $a\in \mathcal{S}$ we have $U(a) = \varnothing$ if and only if for each $a$ we have $\thick^\otimes(a) = \mcK$, where this last equivalence is \cite{BaSpec}*{Corollary~2.5}.
\end{proof}

So restating (\ref{eq_dense}) in terms of Lemma~\ref{lem_cor2.5_bis} we see $\widetilde{F}$ is dense if and only if $\mathcal{S}_\mcL \cie \mathcal{S}_\mcK$. Now applying Lemma \ref{lem_dense_containments} we have proved:

\begin{prop}\label{prop_dense1}
Given a tensor triangulated functor $F\colon \mcK \to \mcL$ the associated map on spectra $\widetilde{F}$ is dense if and only if 
\begin{displaymath}
\mathcal{S}_{\mcL} = \{a\in \mcK \; \vert \; \thick^\otimes(Fa) = \mcL\} \subseteq  \{a\in \mcK \; \vert \; \thick^\otimes(a) = \mcK\} = \mathcal{S}_{\mcK}
\end{displaymath}
and in this case $\mathcal{S}_\mcK = \mathcal{S}_\mcL$.
\end{prop}

\bibliography{greg_bib}

\end{document}